\pdfoutput=1
\RequirePackage{ifpdf}
\ifpdf 
\documentclass[pdftex]{sigma}
\else
\documentclass{sigma}
\fi

\usepackage{tikz}
\usepackage{tikzscale}

\def\Xint#1{\mathchoice
{\XXint\displaystyle\textstyle{#1}}%
{\XXint\textstyle\scriptstyle{#1}}%
{\XXint\scriptstyle\scriptscriptstyle{#1}}%
{\XXint\scriptscriptstyle\scriptscriptstyle{#1}}%
\!\int}
\def\XXint#1#2#3{{\setbox0=\hbox{$#1{#2#3}{\int}$}
\vcenter{\hbox{$#2#3$}}\kern-.5\wd0}}
\def\dashint{\Xint-}
\numberwithin{equation}{section}

\newtheorem{Theorem}{Theorem}[section]
\newtheorem{Corollary}[Theorem]{Corollary}
\newtheorem{Lemma}[Theorem]{Lemma}
 { \theoremstyle{definition}
\newtheorem{Definition}[Theorem]{Definition}
\newtheorem{Remark}[Theorem]{Remark} }

\begin{document}

\allowdisplaybreaks

\newcommand{\arXivNumber}{1705.10544}

\renewcommand{\PaperNumber}{006}

\FirstPageHeading

\ShortArticleName{On the TASEP with Second Class Particles}

\ArticleName{On the TASEP with Second Class Particles}

\Author{Eunghyun LEE}

\AuthorNameForHeading{E.~Lee}

\Address{Department of Mathematics, Nazarbayev University, Kazakhstan}
\Email{\href{mailto:eunghyun.lee@nu.edu.kz}{eunghyun.lee@nu.edu.kz}}
\URLaddress{\url{https://sites.google.com/a/nu.edu.kz/eunghyun-lee-s-homepage/}}

\ArticleDates{Received August 08, 2017, in f\/inal form January 08, 2018; Published online January 12, 2018}

\Abstract{In this paper we study some conditional probabilities for the totally asymmetric simple exclusion processes (TASEP) with second class particles. To be more specif\/ic, we consider a f\/inite system with one f\/irst class particle and $N-1$ second class particles, and we assume that the f\/irst class particle is initially at the leftmost position. In this case, we f\/ind the probability that the f\/irst class particle is at $x$ and it is still the leftmost particle at time $t$. In particular, we show that this probability is expressed by the determinant of an $N\times N$ matrix of contour integrals if the initial positions of particles satisfy the \textit{step initial condition}. The resulting formula is very similar to a known formula in the (\textit{usual}) TASEP with the \textit{step initial condition} which was used for asymptotics by Nagao and Sasamoto [\textit{Nuclear Phys.~B} \textbf{699} (2004), 487--502].}


\Classification{60J25; 60K35; 82B23}

\section{Introduction}
In this paper we study some conditional probabilities in the totally asymmetric simple exclusion processes (TASEP) with second class particles. So far, there have been many works on the coordinate Bethe Ansatz applicable stochastic particle models, but most of them are on the models of a single species such as the asymmetric simple exclusion processes (ASEP) \cite{Schutz-1997,Tracy-Widom-2008}, the $q$-totally asymmetric simple exclusion processes ($q$-TASEP) \cite{Borodin-Corwin-Sasamoto-2014}, the $q$-Hahn asymmetric exclusion processes \cite{Barraquand-Corwin-2016,Lee-2012} and the $q$-totally asymmetric zero range processes ($q$-TAZRP) \cite{Korhonen-Lee-2014,Lee-Wang-2017-arX,Povolotsky-2013,Wang-Waugh-2016}. Our interests in these models include but are not limited to the transition probabilities, the probability distribution of a tagged particle's position for some special initial conditions, and their asymptotics that may conf\/irm that the models belong to the KPZ universality class. However, in this direction of studies, the TASEP with second class particles have not been explored much, and only a limited number of papers were published \cite{Chatterjee-Schutz-2010,Tracy-Widom-2009,Tracy-Widom-2013} because of the complexity in computations. The transition probabilities of the TASEP with second class particles were studied by Chatterjee and Sch\"{u}tz \cite{Chatterjee-Schutz-2010} and also studied in more general setting (the ASEP with multiple species particles) by Tracy and Widom \cite{Tracy-Widom-2013}. A~motivation of this paper starts from naive questions that \textit{``Can we extend the results in the TASEP $($such as the one point distributions for some initial conditions and their asymptotics$)$ to the TASEP with second class particles?'' ``Furthermore, what observables are meaningful and should be considered in connection to the KPZ universality class?''} In this paper, we do not provide complete answers to these questions but would like to study some explicit and exact formulas that can serve as a~starting point for further studies. The integrability of the TASEP with second class particles has been conf\/irmed in \cite{Chatterjee-Schutz-2010,Tracy-Widom-2013}, hence as the next step, we are interested in f\/inding a \textit{neat} probability formula of a certain event from which we hope to obtain a meaningful asymptotic result. We provide a~probability formula (\ref{530-pm-517}) whose form is very similar to a known formula in the TASEP without second class particles. But, it is not clear at this moment whether this high similarity in the appearances will imply easy asymptotics or not. Although we consider the TASEP with second class particles for algebraic simplicity, it seems that the techniques used in this paper can be also used to extend the results in this paper to the ASEP with second class particles. The author conf\/irmed this extension for some \textit{small} systems of the ASEP with second class particles.

 The def\/inition of the TASEP with second class particles is as follows. Each site on $\mathbb{Z}$ can be occupied by at most one particle and each particle belongs to one of two dif\/ferent species, labeled~$1$ or~$2$. A particle of species $i$ at $x\in \mathbb{Z}$ tries to jump to $x+1$ after a waiting time exponentially distributed with rate~1. If the target site $x+1$ is empty, the particle jumps to~\smash{$x+1$}. But, if $x+1$ is not empty, one of the following cases occurs: (i) if $x+1$ is occupied by another particle of the same species $i$, the particle of species~$i$ at $x$ cannot jump to $x+1$, and the waiting time is reset, (ii) if $x$ is occupied by a particle of species~2 and $x+1$ is occupied by a particle of species~1, then the particle of species~2 at $x$ can jump to $x+1$ by interchanging the positions with the particle of species~1 at~$x+1$, (iii) if~$x$ is occupied by a particle of species~1 and~\smash{$x+1$} is occupied by a particle of species~2, then the particle of species~1 cannot jump to~\smash{$x+1$} and the waiting time is reset (see Fig.~\ref{Fig1}).
 \begin{figure}[h]\centering
\begin{tikzpicture}[scale=0.8]
\draw (-0.5,-3)--(1.5,-3)
 (-0.5,-1.5)--(1.5,-1.5)
 (3,-1.5)--(5,-1.5)
 (-0.5,0)--(1.5,0)
 (3,0)--(5,0);
\draw [fill] (0,-3) circle [radius=0.1];
\draw [fill] (1,-3) circle [radius=0.1];
\draw [fill] (0,-1.5) circle [radius=0.1];
\draw [fill] (1,-1.5) circle [radius=0.1];
\draw [fill] (3.5,-1.5) circle [radius=0.1];
\draw [fill] (4.5,-1.5) circle [radius=0.1];
\draw [fill] (0,0) circle [radius=0.1];
\draw [fill] (1,0) circle [radius=0.1];
\draw [fill] (3.5,0) circle [radius=0.1];
\draw [fill] (4.5,0) circle [radius=0.1];
\node[below] at (0,-3.1) {$1$};
\node[below] at (1,-3.1) {$2$};
\node[below] at (0,-1.6) {$2$};
\node[below] at (1,-1.6) {$1$};
\node[below] at (3.5,-1.6) {$1$};
\node[below] at (4.5,-1.6) {$2$};
\node[below] at (0,-0.1) {$1$};
\node[below] at (1,-0.1) {$1$};
\node[below] at (3.5,-0.1) {$2$};
\node[below] at (4.5,-0.1) {$2$};
\node at (2.25,-1.5) {$\Longrightarrow$};
\draw[->](0,-2.8) arc [radius=0.55, start angle=150, end angle= 30];
\draw[->](0,-1.3) arc [radius=0.55, start angle=150, end angle= 30];
\draw[->](0,0.2) arc [radius=0.55, start angle=150, end angle= 30];
\draw[->](3.5,0.2) arc [radius=0.55, start angle=150, end angle= 30];
\node at (0.5,-2.53) {$\times$};
\node at (0.5,0.47) {$\times$};
\node at (4,0.47) {$\times$};
\node[align=left] at (8,0) {(i)};
\node[align=left] at (8,-1.5) {(ii)};
\node[align=left] at (8,-3) {(iii)};
\end{tikzpicture}
\caption{}\label{Fig1}
\end{figure}

In other words, the particles of species 2 have a priority over those of species 1. The particles of species 1 are called second class particles and the particles of species~2 are called f\/irst class particles.
\subsection{Statement of the main results}In this paper we consider a f\/inite system with~$N$ particles. The state space of the process is a~countable set of pairs $(X,\pi)$ where $X=(x_1,\dots,x_N) \in \mathbb{W}^N$ with
\begin{gather*}
\mathbb{W}^N = \big\{(x_1,\dots, x_N)\colon (x_1,\dots, x_N) \in \mathbb{Z}^N ~\textrm{and}~ x_1 < \cdots < x_N\big\}
\end{gather*}
and $\pi=(\pi_1\cdots\pi_N)$ is a f\/inite sequence of length $N$ whose elements are 1 or 2. A state
\begin{gather*}
\big((x_1,\dots, x_N),(\pi_1\cdots\pi_N)\big)
\end{gather*}
implies that the $i^{\textrm{th}}$ particle from the left is at $x_i$ and this particle belongs to species $\pi_i$. We will sometimes omit the parentheses in the expressions of $(x_1,\dots,x_N)$ and $(\pi_1\cdots\pi_N)$. If we would like to express a state at time $t$, we write
\begin{gather*}
\big((x_1,\dots, x_N),(\pi_1\cdots\pi_N);t\big)
\end{gather*}
 We denote by $P_{(Y,\nu)}(X,\pi;t)$ a transition probability, that is, the probability that the system is at state $(X,\pi)$ at time $t$, given that the initial state is $(Y,\nu)$. We denote by $\mathbb{P}_{(Y,\nu)}$ the probability measure of the process with initial state $(Y,\nu)$. Let $E_t=\{(X,21\dots 1;t)\colon x_1 = x\}$, the event that the leftmost particle is the f\/irst class particle and it is at $x$ at time $t$.
 \begin{Remark} One of the motivations of this paper is the fact that if $\pi=\nu$, then the transition probabilities $P_{(Y,\nu)}\big((x_1,\dots,x_N),\nu;t\big)$ can be expressed as a determinant of an $N\times N$ matrix \cite{Chatterjee-Schutz-2010}. In the TASEP with $N$ particles of the same kind, the transition probabilities are also expressed as a determinant \cite{Schutz-1997} and the distribution of the leftmost particle's position is obtained by summing the determinants over all possible $x_2,\dots,x_N$ with $x_1<x_2<\cdots<x_N$ for f\/ixed~$x_1$. This multiple sum of the determinants can be expressed as a closed form of a multiple contour integral when the initial positions of particles are given by $(1,2,\dots, N)$ \cite{Nagao-Sasamoto-2004,Tracy-Widom-2008}. This was possible due to some special properties of the determinants \cite[Theorem~4]{Nagao-Sasamoto-2004} or equivalently due to the algebraic identity which was found in more general setting in \cite[equation~(1.6)]{Tracy-Widom-2008}. Hence, it is natural to ask if there is a similar algebraic structure originated from the determinantal form of~$P_{(Y,\nu)}(X,\nu;t)$ in the TASEP with second class particles, and (if so) whether or not it is possible to express a multiple sum of $P_{(Y,\nu)}(X,\nu;t)$ over $x_2,\dots,x_N$ with $x_1<x_2<\cdots<x_N$ for f\/ixed~$x_1$ as a closed form of a multiple contour integral.
 \end{Remark}

Throughout this paper, we will use the notation $\dashint$ for $\frac{1}{2\pi i}\int$. The f\/irst result provides an answer to the question \textit{``if the first class particle is initially the leftmost particle, what is the probability that the first particle moves to $x$ during time $t$ without exchanging positions with second class particles?"}
 \begin{Theorem}\label{142-am-518}
 Let $C$ be a counterclockwise circle centered at the origin with radius less than~$1$. Then,
\begin{gather}
\mathbb{P}_{(Y,21\dots 1)}(E_t) = \dashint_C\cdots\dashint_C(1-\xi_1)\prod_{1\leq i<j\leq N}\frac{\xi_j-\xi_i}{1-\xi_i}\nonumber\\
\hphantom{\mathbb{P}_{(Y,21\dots 1)}(E_t) =}{} \times\prod_{i=1}^N\frac{1}{1-\xi_i}\prod_{i=1}^N \big(\xi_i^{x-y_i-1}e^{\varepsilon(\xi_i) t}\big){\rm d}\xi_1\cdots {\rm d}\xi_N,\label{125-am-518}
\end{gather}
where $\varepsilon(\xi_i) = 1/\xi_i - 1$.
\end{Theorem}

 This formula is comparable with that of the TASEP (with a single species). The formula for the TASEP corresponding to (\ref{125-am-518}) is
\begin{gather*}
\mathbb{P}_{(Y,i\dots i)}(F_t) = \dashint_C\cdots\dashint_C(1-\xi_1\cdots\xi_N)\prod_{1\leq i<j\leq N}\frac{\xi_j-\xi_i}{1-\xi_i}\nonumber\\
\hphantom{\mathbb{P}_{(Y,i\dots i)}(F_t) =}{}\times \prod_{i=1}^N\frac{1}{1-\xi_i}\prod_{i=1}^N \big(\xi_i^{x-y_i-1}e^{\varepsilon(\xi_i) t}\big){\rm d}\xi_1\cdots {\rm d}\xi_N,\label{1217-pm-522}
\end{gather*}
where $F_t=\{(X,i\dots i;t)\colon x_1 = x\}$ \cite{Tracy-Widom-2008}.
\begin{Remark}\label{124810}
In order to obtain (\ref{125-am-518}), we need to f\/ind an expression of $P_{(Y,21\dots 1)}(X,21\dots 1;t)$. One novel point of this paper is that we provide an explicit and exact formula of the transition probability $P_{(Y,21\dots 1)}(X,21\dots 1;t)$. Chatterjee and Sch\"{u}tz \cite{Chatterjee-Schutz-2010}, and Tracy and Widom \cite{Tracy-Widom-2013} discussed a method to obtain $P_{(Y,\nu)}(X,\pi;t)$ but did not provide an explicit formula for specif\/ic~$\nu$ and~$\pi$. The transition probabilities of an $N$-particle system (with or without second class par\-tic\-les) are expressed in a compact matrix form of contour integrals as in (\ref{1223-am-519}), and each matrix element is a specif\/ic transition probability. Finding an explicit formula of a matrix element of~$\mathbf{P}_Y(X;t)$ in (\ref{1223-am-519}) means f\/inding an explicit formula of a matrix element of~$\mathbf{A}_{\sigma}$ in~(\ref{1223-am-519}). But this may require a tedious matrix multiplication of several $2^N$ by $2^N$ matrices. To the best of the author's knowledge, there are no known shortcut methods to f\/ind each matrix element of~$\mathbf{A}_{\sigma}$. Lemma \ref{146-am-520} is an important result to provide an explicit formula of $P_{(Y,21\dots 1)}(X,21\dots 1;t)$.
 \end{Remark}

 The second result is obtained by applying a special initial condition for the positions of particles to (\ref{125-am-518}). Let $Y=(y_1,\dots, y_N)$ be the initial positions of particles given by
\begin{gather}\label{107-am-518}
y_i = \begin{cases}
1& \textrm{if}~i=1,\\
i+l& \textrm{if}~i>1
\end{cases}
\end{gather}
for some nonnegative integer $l$. When $l=0$, this condition is called the \textit{step initial condition}. Let us recall that
\begin{gather*}
h_l(\xi_1,\dots,\xi_N) = \sum_{1\leq i_i \leq \cdots \leq i_l \leq N}\xi_{i_1}\xi_{i_2}\cdots \xi_{i_l}
\end{gather*}
is called the $l$th complete symmetric polynomial of $\xi_1,\dots, \xi_N$, and $h_0(\xi_1,\dots,\xi_N)=1$.
\begin{Theorem}\label{144-am-518}
If $Y$ is given by \eqref{107-am-518},
\begin{gather}
 \mathbb{P}_{(Y,21\dots 1)}(E_t) = \frac{(-1)^{N(N-1)/2}}{N!} \dashint_C\cdots\dashint_Ch_l(\xi_1,\dots,\xi_N)\prod_{1\leq i<j\leq N}(\xi_j-\xi_i)^2\prod_{i=1}^N\frac{1}{(\xi_i-1)^{N-1}}\nonumber\\
\hphantom{\mathbb{P}_{(Y,21\dots 1)}(E_t) =}{} \times\prod_{i=1}^N \big(\xi_i^{x-N-l-1}e^{\varepsilon(\xi_i) t}\big){\rm d}\xi_1\cdots {\rm d}\xi_N.\label{534-pm-517}
\end{gather}
\end{Theorem}

If $Y$ is given by the step initial condition, then (\ref{534-pm-517}) is very similar to a known formula in the TASEP
\cite[Remark on p.~839]{Tracy-Widom-2008}. (Also, see \cite{Johansson-2000,Nagao-Sasamoto-2004} for similar results.)
\begin{Corollary}\label{143-am-519}
If $Y$ is given by \eqref{107-am-518} with $l=0$, then
\begin{gather}
 \mathbb{P}_{(Y,21\dots 1)}(E_t) = \frac{(-1)^{N(N-1)/2}}{N!} \dashint_C\cdots\dashint_C\prod_{1\leq i<j\leq N}(\xi_j-\xi_i)^2\prod_{i=1}^N\frac{1}{(\xi_i-1)^{N-1}}\nonumber\\
\hphantom{\mathbb{P}_{(Y,21\dots 1)}(E_t) =}{}\times\prod_{i=1}^N \big(\xi_i^{x-N-1}e^{\varepsilon(\xi_i) t}\big){\rm d}\xi_1\cdots {\rm d}\xi_N.\label{530-pm-517}
\end{gather}
\end{Corollary}

Also, (\ref{530-pm-517}) can be written as
\begin{gather*}
 \mathbb{P}_{(Y,21\dots 1)}(E_t) = (-1)^{N(N-1)/2}\det\Bigg[\dashint_C\xi^{i+j+x-N-1}(\xi-1)^{-(N-1)}e^{(1/\xi-1)t}{\rm d}\xi \Bigg]_{i,j=0}^{N-1}
\end{gather*}
by the same method as in \cite[Remark on p.~839]{Tracy-Widom-2008}.
\begin{Remark}
 It would be interesting to see if $\mathbb{P}_{(Y,\nu)}(E_t^{\pi})$ where $E_t^{\pi} = \{(X,\pi;t)\colon x_1=x\}$ for arbitrary $\nu$ and $\pi$ can be expressed as a determinant. To do so, at f\/irst, we need to f\/ind an explicit formula of $P_{(Y,\nu)}(X,\pi;t)$. However, as mentioned in Remark~\ref{124810}, there are no known methods to f\/ind $P_{(Y,\nu)}(X,\pi;t)$ except a tedius matrix computation. A~partial generalization of $\mathbb{P}_{(Y,21\dots 1)}(E_t)$ was made a few months after the f\/irst version of this paper was posted~\cite{Lee-2017-July}.
 \end{Remark}

\subsection{Previous results}
\looseness=-1 There are some previous results on the TASEP (or the ASEP) with a dif\/ferent initial conf\/i\-gu\-ra\-tion. Suppose that initially all negative integers are occupied by f\/irst class particles, the origin is occupied by a second class particle, and all positive integers are empty. In this case, the position of the second class particle at time $t$ is nontrivial because it is af\/fected by f\/irst class particles' movements. Tracy and Widom obtained the exact expression of the distribution of the position of the second class particle at time $t$ in the ASEP \cite{Tracy-Widom-2009}. Mountford and Guiol proved the strong law of large numbers for the position of the second class particle in the TASEP \cite{Mountford-Guiol-2005}. On the other hand, for the initial condition considered in this paper, the probability distribution of the position of the (only) f\/irst class particle at time $t$ can be trivially obtained because it is a free particle if we do not care about the order of particles. But, the probability for the non-intersecting paths of the particles from the initial conf\/iguration to any conf\/igurations such that the leftmost particle (the f\/irst class particle) is at a specif\/ic position at time $t$ is nontrivial as in Theorem~\ref{142-am-518}.
 \begin{Remark}
 The initial conf\/iguration $(Y,\nu)$ for $Y= (\dots,-2,-1)$ and $\nu = (\dots 221\dots 1)$, and the probability formula of the position of the $m^{\textrm{th}}$ rightmost f\/irst class particle at time $t$ are of particular interest. If we do not care about the order of the f\/irst particles and the second class particle at time $t$, the system can be thought of as the TASEP without second class particles because the f\/irst class particles \textit{see} the second class particles as holes. Hence, the distribution of the $m^{\textrm{th}}$ rightmost f\/irst class particle's position is trivially obtained from the previous results of the TASEP and the Tracy--Widom distribution should appear in the asymptotic studies. If we consider an event that the initial order of particles does not change over time, the probability of the $m^{\textrm{th}}$ rightmost f\/irst class particle's position at time $t$ conf\/ined in this even is nontrivial.
\end{Remark}

{\bf Organization of the paper.}
In Section~\ref{section2}, we provide some preliminary background of the integrability and the transition probabilities of the TASEP with second class particles. Although the majority of Section~\ref{section2} is essentially overlapped with \cite{Chatterjee-Schutz-2010,Tracy-Widom-2013}, we introduce this background for the notations in this paper and self-containedness. In Section~\ref{section3}, we derive (\ref{125-am-518}), (\ref{534-pm-517}), and~(\ref{530-pm-517}).

\section{Preliminary}\label{section2}
Since the state space of the process is countable, we may view $P_{(Y,\nu)}(X,\pi;t)$ as elements of an inf\/inite matrix. This inf\/inite matrix is denoted by $\mathbf{P}(t)$ and we assume that $P_{(Y,\nu)}(X,\pi;t)$ is an element at the $(Y,\nu)^{\textrm{th}}$ column and at the $(X,\pi)^{\textrm{th}}$ row. We denote a submatrix of $\mathbf{P}(t)$ with elements $P_{(Y,\nu)}(X,\pi;t)$ for f\/ixed $X$ and $Y$ as $\mathbf{P}_Y(X;t)$. Hence, $\mathbf{P}_Y(X;t)$ is a $2^N \times 2^N$ matrix. We assume that the elements of $\mathbf{P}_Y(X;t)$ are listed in the reverse lexicographic order of f\/inite sequences $\nu$ and $\pi$ from left to right and from top to bottom, respectively (See (\ref{424am58}) below for an example). If $\mathbf{G}$ is the generator of the process, then $\mathbf{P}(t)$ satisf\/ies the forward equation
\begin{gather}
\frac{{\rm d}}{{\rm d}t}\mathbf{P}(t) = \mathbf{G}\mathbf{P}(t) \label{1141-223}
\end{gather}
and it is subject to the initial condition $\mathbf{P}(0) = \mathbf{I}$ where $\mathbf{I}$ is the identity matrix.
The def\/inition of the TASEP with second class particles implies that $\mathbf{G}=[a_{jk}]_{j,k
\in \mathbb{Z}}$ is a band matrix with $\sup\limits_{j,k}|a_{jk}| = N<
\infty$ (so, $\mathbf{G}$ induces a bounded operator on $l_2(\mathbb{Z})$), and the solution of (\ref{1141-223}) is simply given by $\mathbf{P}(t)=e^{t\mathbf{G}}$. We want to f\/ind an explicit formula of each matrix element of $e^{t\mathbf{G}}$, that is, the transition probabilities $P_{(Y,\nu)}(X,\pi;t)$ by using the standard method, the coordinate Bethe Ansatz \cite{Schutz-1997,Tracy-Widom-2008}.
\subsection{Two-particle system}
By the def\/inition of the TASEP with second class particles, (\ref{1141-223}) implies that a submatrix $\mathbf{P}_Y(x_1,x_2;t)$ given by
\begin{gather}
\left[
 \begin{matrix}
 {P}_{(Y,11)}(x_1,x_2,11;t) & {P}_{(Y,12)}(x_1,x_2,11;t) & {P}_{(Y,21)}(x_1,x_2,11;t) & {P}_{(Y,22)}(x_1,x_2,11;t) \\
 {P}_{(Y,11)}(x_1,x_2,12;t) & {P}_{(Y,12)}(x_1,x_2,12;t) & {P}_{(Y,21)}(x_1,x_2,12;t) & {P}_{(Y,22)}(x_1,x_2,12;t) \\
 {P}_{(Y,11)}(x_1,x_2,21;t) & {P}_{(Y,12)}(x_1,x_2,21;t) & {P}_{(Y,21)}(x_1,x_2,21;t) & {P}_{(Y,22)}(x_1,x_2,21;t) \\
 {P}_{(Y,11)}(x_1,x_2,22;t) & {P}_{(Y,12)}(x_1,x_2,22;t) & {P}_{(Y,21)}(x_1,x_2,22;t) & {P}_{(Y,22)}(x_1,x_2,22;t) \\
 \end{matrix}
 \right]\!\!\!\! \label{424am58}
\end{gather}
 satisf\/ies one of the following equations, depending on $(x_1,x_2)$:
\begin{gather}
\frac{{\rm d}}{{\rm d}t} \mathbf{P}_{Y}(x_1,x_2;t) = \mathbf{P}_{Y}(x_1-1,x_2;t) + \mathbf{P}_{Y}(x_1,x_2-1;t) \nonumber\\
\hphantom{\frac{{\rm d}}{{\rm d}t} \mathbf{P}_{Y}(x_1,x_2;t) =}{} - 2\mathbf{P}_{Y}(x_1,x_2;t), \qquad x_1 <x_2-1, \label{300am58}\\
\frac{{\rm d}}{{\rm d}t} \mathbf{P}_{Y}(x_1,x_2;t) = \mathbf{P}_{Y}(x_1-1,x_2;t) - \left[
 \begin{matrix}
 1 & 0 & 0 & 0 \\
 0 & 1 & -1 & 0 \\
 0 & 0 & 2 & 0 \\
 0 & 0 & 0 & 1 \\
 \end{matrix}
 \right]\mathbf{P}_{Y}(x_1,x_2;t), \qquad x_1 = x_2-1.\!\!\!\! \label{1231-224}
\end{gather}
Let $\mathbf{U}(x_1,x_2;t)$ be a $4 \times 4$ matrix whose elements are functions def\/ined on $\mathbb{Z}^2 \times [0,\infty)$. If we suppose that $\mathbf{U}(x_1,x_2;t)$ is a solution of
\begin{gather} \label{eq1}
\frac{{\rm d}}{{\rm d}t} \mathbf{U}(x_1,x_2;t) = \mathbf{U}(x_1-1,x_2;t) + \mathbf{U}(x_1,x_2-1;t) - 2\mathbf{U}(x_1,x_2;t)
\end{gather}
for all $(x_1,x_2) \in \mathbb{Z}^2$ and is subject to
\begin{gather*}
\mathbf{U}(x,x;t) -2\mathbf{U}(x,x+1;t) = -\left[
 \begin{matrix}
 1 & 0 & 0 & 0 \\
 0 & 1 & -1 & 0 \\
 0 & 0 & 2 & 0 \\
 0 & 0 & 0 & 1 \\
 \end{matrix}
 \right]
\mathbf{U}(x,x+1;t)\qquad \text{for all}\quad x \in \mathbb{Z},
\end{gather*}
so that
\begin{gather}
\mathbf{U}(x,x;t) = \left[
 \begin{matrix}
 1 & 0 & 0 & 0 \\
 0 & 1 & 1 & 0 \\
 0 & 0 & 0 & 0 \\
 0 & 0 & 0 & 1 \\
 \end{matrix}
 \right]
\mathbf{U}(x,x+1;t):=\mathbf{B}\mathbf{U}(x,x+1;t), \label{22519}
\end{gather}
then it is obvious that $\mathbf{U}(x_1,x_2;t)$ for $ x_1<x_2-1$ satisf\/ies (\ref{300am58}), and both (\ref{eq1}) and (\ref{22519}) imply that $\mathbf{U}(x_1,x_2;t)$ for $x=x_1 = x_2-1$ satisf\/ies~(\ref{1231-224}). An ansatz of variables separation, $\mathbf{U}(x_1,x_2;t) = \mathbf{X}(x_1,x_2)T(t)$ where $\mathbf{X}(x_1,x_2)$ is a $4\times 4$ matrix and $T(t)$ is a scalar function of~$t$, reduces the matrix equation (\ref{eq1}) to a linear matrix recurrence relation of two variables
\begin{gather}
\varepsilon \mathbf{X}(x_1,x_2) = \mathbf{X}(x_1-1,x_2) + \mathbf{X}(x_1,x_2-1) - 2\mathbf{X}(x_1,x_2) \label{1249-am-523}
\end{gather}
and $T(t) = e^{\varepsilon t}$ where $\varepsilon$ is to be determined. By an ansatz from the theory of the linear recurrence relation we see that $\mathbf{I}_4\xi_1^{x_1}\xi_2^{x_2}$ is a solution of (\ref{1249-am-523}) where $\mathbf{I}_4$ is the $4\times 4$ identity matrix (we will write $\mathbf{I}_n$ for the $n \times n$ identity matrix), and $\xi_1$ and $\xi_2$ are complex numbers with $0<|\xi_1|<1$ and $0<|\xi_2|<1$ (the reason for this restriction on $\xi_1$ and $\xi_2$ will be clear soon), and in this case, $\varepsilon$ is given by
\begin{gather}
\varepsilon = \frac{1}{\xi_1} + \frac{1}{\xi_2} - 2. \label{1247-am-523}
\end{gather}
In fact, by the linearity of the equation, for any $4 \times 4$ matrix $\mathbf{A}_{12}$, independent of $x_1$ and $x_2$, $\mathbf{A}_{12}\xi_1^{x_1}\xi_2^{x_2}$ is also a solution of (\ref{1249-am-523}). Also, we observe that $\mathbf{I}_4\xi_2^{x_1}\xi_1^{x_2}$ is a solution of (\ref{1249-am-523}) with the same $\varepsilon$ in (\ref{1247-am-523}). Thus, for a general solution of (\ref{eq1}), we put
\begin{gather}
\mathbf{U}(x_1,x_2;t) = \big(\mathbf{A}_{12}\xi_1^{x_1}\xi_2^{x_2} +\mathbf{A}_{21}\xi_2^{x_1}\xi_1^{x_2}\big)e^{\varepsilon t}, \label{22518}
\end{gather}
where $\mathbf{A}_{21}$ is another $4 \times 4$ matrix, independent of $x_1$ and $x_2$.
Substituting (\ref{22518}) into (\ref{22519}), we obtain
\begin{gather*}
(\mathbf{I}_4-\xi_2\mathbf{B})\mathbf{A}_{12} = -(\mathbf{I}_4-\xi_1\mathbf{B})\mathbf{A}_{21},
\end{gather*}
 and
\begin{gather}
\mathbf{A}_{21} = -(\mathbf{I}_4-\xi_1\mathbf{B})^{-1}(\mathbf{I}_4-\xi_2\mathbf{B})\mathbf{A}_{12} = \left[
 \begin{matrix}
 -\frac{1-\xi_2}{1-\xi_1} & 0 & 0 & 0 \\
0 & -\frac{1-\xi_2}{1-\xi_1} & \frac{\xi_2-\xi_1}{1-\xi_1} & 0 \\
 0 & 0 & -1 & 0 \\
 0 & 0 & 0 & -\frac{1-\xi_2}{1-\xi_1} \\
 \end{matrix}
 \right]\mathbf{A}_{12}.\label{102224}
\end{gather}
\begin{Definition}
\begin{gather}
\mathbf{S}_{\beta\alpha} := \left[
 \begin{matrix}
 -\frac{1-\xi_{\beta}}{1-\xi_{\alpha}} & 0 & 0 & 0 \\
0 & -\frac{1-\xi_{\beta}}{1-\xi_{\alpha}} & \frac{\xi_{\beta}-\xi_{\alpha}}{1-\xi_{\alpha}} & 0 \\
 0 & 0 & -1 & 0 \\
 0 & 0 & 0 & -\frac{1-\xi_{\beta}}{1-\xi_{\alpha}} \\
 \end{matrix}
\right]\qquad \textrm{and} \qquad S_{\beta\alpha}:=-\frac{1-\xi_{\beta}}{1-\xi_{\alpha}}.
 \label{546pm511}
\end{gather}
\end{Definition}

Hence, (\ref{22518}) with (\ref{102224}) satisf\/ies (\ref{300am58}) and (\ref{1231-224}). Next, we put $\mathbf{A}_{12} = \mathbf{I}_4\xi_1^{-y_1-1}\xi_2^{-y_2-1}$ for $Y=(y_1,y_2)$ in (\ref{22518}) and integrate componentwise over counterclockwise circles $C$ centered at the origin with radii less than 1, to form a matrix of contour integrals written
\begin{gather}\label{123836}
 \dashint_C\dashint_C \big( \xi_1^{x_1-y_1-1}\xi_2^{x_2-y_2-1}\mathbf{I}_4
 + \xi_2^{x_1-y_2-1}\xi_1^{x_2-y_1-1}\mathbf{S}_{21}\big) e^{\varepsilon t}{\rm d}\xi_1{\rm d}\xi_2.
\end{gather}
Recall that the matrix (\ref{424am58}) satisf\/ies the initial condition which states that at $t=0$, if $x_1= y_1$ and $x_2 = y_2$, then the matrix (\ref{424am58}) is the identity matrix and, otherwise, it is the zero matrix. Recalling that $x_1 \geq y_1$, $x_2 \geq y_2$, $x_1<x_2,$ and $y_1<y_2$, we can show that (\ref{123836}) satisf\/ies this initial condition. Hence,
\begin{gather*}
\mathbf{P}_{Y}(x_1,x_2;t)= \dashint_C\dashint_C \big( \xi_1^{x_1-y_1-1}\xi_2^{x_2-y_2-1}\mathbf{I}_4
 + \xi_2^{x_1-y_2-1}\xi_1^{x_2-y_1-1}\mathbf{S}_{21}\big) e^{\varepsilon t}{\rm d}\xi_1{\rm d}\xi_2.
\end{gather*}

\subsection{Three-particle system}
Let $\mathbf{U}(x_1,x_2,x_3;t)$ be an $8 \times 8$ matrix whose elements are functions def\/ined on $\mathbb{Z}^3 \times [0,\infty)$. Suppose that $\mathbf{U}(x_1,x_2,x_3;t)$ is a solution of
\begin{gather}
\frac{{\rm d}}{{\rm d}t} \mathbf{U}(x_1,x_2,x_3;t) = \mathbf{U}(x_1-1,x_2,x_3;t) + \mathbf{U}(x_1,x_2-1,x_3;t)\nonumber\\
\hphantom{\frac{{\rm d}}{{\rm d}t} \mathbf{U}(x_1,x_2,x_3;t) =}{} + \mathbf{U}(x_1,x_2-1,x_3-1;t) - 3\mathbf{U}(x_1,x_2,x_3;t)\label{552-224}
\end{gather}
and is subject to
\begin{gather}
\mathbf{U}(x,x,x_3;t) = (\mathbf{B} \otimes \mathbf{I}_{2}) \mathbf{U}(x,x+1,x_3;t)\qquad \text{for all} \quad x,x_3 \in \mathbb{Z}, \nonumber\\
\mathbf{U}(x_1,x,x;t) = (\mathbf{I}_{2} \otimes \mathbf{B} ) \mathbf{U}(x_1,x,x+1;t) \qquad \text{for all} \quad x_1,x \in \mathbb{Z},\label{1123-225}
\end{gather}
where $\otimes$ means the tensor product of matrices. Then, as in $N=2$, it is possible to show that for each $(x_1,x_2,x_3) \in \mathbb{W}^3$, $\mathbf{U}(x_1,x_2,x_3;t)$ satisf\/ies the dif\/ferential equation for $\mathbf{P}_Y(x_1,x_2,x_3;t)$. As the Bethe Ansatz solution of (\ref{552-224}), we put
\begin{gather}\label{1118-225}
\mathbf{U}(x_1,x_2,x_3;t) = \sum_{\sigma\in {S}_3}\mathbf{A}_{\sigma}\xi_{\sigma(1)}^{x_1}\xi_{\sigma(2)}^{x_2}\xi_{\sigma(3)}^{x_3}e^{\varepsilon t} ,
\end{gather}
where
\begin{gather*}
\varepsilon = \frac{1}{\xi_1} + \frac{1}{\xi_2} + \frac{1}{\xi_3} - 3
\end{gather*}
and $\xi_i,(i=1,2,3)$ are nonzero complex numbers with $0<|\xi_i|<1$, and $\mathbf{A}_{\sigma}$ are $8 \times 8$ matrices of complex numbers independent of $x_1$, $x_2$, $x_3$. Here, the sum is over all permutations $\sigma$ in the symmetric group ${S}_3$. Substituting (\ref{1118-225}) into (\ref{1123-225}), we obtain
\begin{gather}
\mathbf{A}_{213} = -(\mathbf{I}_8 - \mathbf{B}\otimes \mathbf{I}_2\xi_1)^{-1}(\mathbf{I}_8 - \mathbf{B}\otimes \mathbf{I}_2\xi_2)\mathbf{A}_{123}, \nonumber\\
\mathbf{A}_{312} = -(\mathbf{I}_8 - \mathbf{B}\otimes \mathbf{I}_2\xi_1)^{-1}(\mathbf{I}_8 - \mathbf{B}\otimes \mathbf{I}_2\xi_3)\mathbf{A}_{132}, \nonumber\\
\mathbf{A}_{321} = -(\mathbf{I}_8 - \mathbf{B}\otimes \mathbf{I}_2\xi_2)^{-1}(\mathbf{I}_8 - \mathbf{B}\otimes \mathbf{I}_2\xi_3)\mathbf{A}_{231},\label{316am59}
\end{gather}
and
\begin{gather}
\mathbf{A}_{132} = -(\mathbf{I}_8 - \mathbf{I}_2 \otimes \mathbf{B}\xi_2)^{-1}(\mathbf{I}_8 - \mathbf{I}_2 \otimes \mathbf{B}\xi_3)\mathbf{A}_{123},\nonumber \\
 \mathbf{A}_{231} = -(\mathbf{I}_8 - \mathbf{I}_2 \otimes \mathbf{B}\xi_1)^{-1}(\mathbf{I}_8 - \mathbf{I}_2 \otimes \mathbf{B}\xi_3)\mathbf{A}_{213}, \nonumber\\
 \mathbf{A}_{321} = -(\mathbf{I}_8 - \mathbf{I}_2 \otimes \mathbf{B}\xi_1)^{-1}(\mathbf{I}_8 - \mathbf{I}_2 \otimes \mathbf{B}\xi_2)\mathbf{A}_{312}. \label{317am59}
\end{gather}
Using the properties of the tensor product of (invertible) linear operators
\begin{gather*}
 (\mathbf{A}\otimes \mathbf{B})(\mathbf{C} \otimes \mathbf{D}) = \mathbf{AC}\otimes \mathbf{BD},\qquad
(\mathbf{A}\otimes \mathbf{B})^{-1} = \mathbf{A}^{-1}\otimes \mathbf{B}^{-1},
\end{gather*}
and recalling the notation of $\mathbf{S}_{\beta\alpha}$ in (\ref{546pm511}), we have
\begin{gather}
\mathbf{I}_2 \otimes \mathbf{S}_{\beta\alpha} = -\mathbf{I}_2 \otimes(\mathbf{I}_4 - \mathbf{B}\xi_{\alpha})^{-1}(\mathbf{I}_4 - \mathbf{B}\xi_{\beta}) = -(\mathbf{I}_8 - \mathbf{I}_2 \otimes \mathbf{B}\xi_{\alpha})^{-1}(\mathbf{I}_8 - \mathbf{I}_2 \otimes \mathbf{B}\xi_{\beta}),\label{1228-pm-523}\\
\mathbf{S}_{\beta\alpha} \otimes\mathbf{I}_2= -(\mathbf{I}_4 - \mathbf{B}\xi_{\alpha})^{-1}(\mathbf{I}_4 - \mathbf{B}\xi_{\beta})\otimes \mathbf{I}_2 = -(\mathbf{I}_8 - \mathbf{B}\otimes \mathbf{I}_2\xi_{\alpha})^{-1}(\mathbf{I}_8 - \mathbf{B}\otimes \mathbf{I}_2\xi_{\beta}).\label{1228-pm-525}
\end{gather}
Let $\alpha = \sigma(i)$, $\beta = \sigma(i+1)$ and let $T_i \in {S}_n$ be a simple transposition for any symmetric group~$S_n$ such that $(T_i\sigma)(i)=\sigma(i+1)$ and $(T_i\sigma)(i+1)=\sigma(i)$, and $(T_i\sigma)(k)=\sigma(k)$ for $k\neq i,i+1$. Then, (\ref{316am59}) and (\ref{317am59}) can be simply written as
\begin{gather*}
\mathbf{A}_{T_1\sigma} = (\mathbf{S}_{\beta\alpha} \otimes\mathbf{I}_2)\mathbf{A}_{\sigma},\qquad
\mathbf{A}_{T_2\sigma}= (\mathbf{I}_2 \otimes \mathbf{S}_{\beta\alpha})\mathbf{A}_{\sigma},
\end{gather*}
respectively. It can be shown that (\ref{1228-pm-523}) and (\ref{1228-pm-525}) satisfy the following relations by simple computations
\begin{alignat}{3}
& \textrm{(i)} \quad && (\mathbf{S}_{\gamma\beta} \otimes \mathbf{I}_2)(\mathbf{I}_2 \otimes \mathbf{S}_{\gamma\alpha})(\mathbf{S}_{\beta\alpha} \otimes \mathbf{I}_2) = (\mathbf{I}_2 \otimes \mathbf{S}_{\beta\alpha})(\mathbf{S}_{\gamma\alpha} \otimes \mathbf{I}_2)(\mathbf{I}_2 \otimes \mathbf{S}_{\gamma\beta}), &\label{134-am-529}\\
& \textrm{(ii)}\quad && (\mathbf{S}_{\beta\alpha} \otimes \mathbf{I}_2)(\mathbf{S}_{\alpha\beta} \otimes \mathbf{I}_2) = \mathbf{I}_8=( \mathbf{I}_2 \otimes\mathbf{S}_{\beta\alpha})( \mathbf{I}_2 \otimes\mathbf{S}_{\alpha\beta}).& \label{135-am-529} \end{alignat}
 Hence, we have the following expressions for $\mathbf{A}_{\sigma}$:
\begin{gather*}
\mathbf{A}_{213}= (\mathbf{S}_{21} \otimes\mathbf{I}_2)\mathbf{A}_{123}, \nonumber\\
\mathbf{A}_{132} = (\mathbf{I}_2 \otimes \mathbf{S}_{32})\mathbf{A}_{123},\nonumber \\
\mathbf{A}_{312} = (\mathbf{S}_{31} \otimes\mathbf{I}_2)(\mathbf{I}_2 \otimes \mathbf{S}_{32})\mathbf{A}_{123},\nonumber\\
\mathbf{A}_{231} = (\mathbf{I}_2 \otimes \mathbf{S}_{31})(\mathbf{S}_{21} \otimes\mathbf{I}_2)\mathbf{A}_{123},\nonumber\\
 \mathbf{A}_{321} = (\mathbf{S}_{32} \otimes\mathbf{I}_2)(\mathbf{I}_2 \otimes \mathbf{S}_{31})(\mathbf{S}_{21} \otimes\mathbf{I}_2)\mathbf{A}_{123}.
\end{gather*}
If we put $\mathbf{A}_{123} = \mathbf{I}_8\prod\limits_{i=1}^3\xi_i^{-y_i-1}$ and integrate (\ref{1118-225}) over counterclockwise circles $C$ centered at the origin with radii less than 1, then we obtain
\begin{gather*}
\mathbf{P}_Y(x_1,x_2,x_3;t) = \dashint_C\dashint_C \dashint_C \sum_{\sigma\in {S}_3}\mathbf{A}_{\sigma}\prod_{i=1}^3\xi_{\sigma(i)}^{x_i}e^{\varepsilon t} {\rm d}\xi_1{\rm d}\xi_2{\rm d}\xi_3.
\end{gather*}

\subsection[$N$-particle system]{$\boldsymbol{N}$-particle system}
\subsubsection{Bethe Ansatz solution}
Let $\mathbf{U}(x_1,\dots,x_N;t)$ be a $2^N \times 2^N$ matrix whose elements are functions def\/ined on $\mathbb{Z}^N \times [0,\infty)$. Suppose that $\mathbf{U}(x_1,\dots,x_N;t)$ is a solution of
\begin{gather*}
\frac{{\rm d}}{{\rm d}t} \mathbf{U}(x_1,\dots,x_N;t) = \sum_{i=1}^N\mathbf{U}(x_1,\dots,x_{i-1},x_{i}-1,x_{i+1},\dots,x_N;t) - N\mathbf{U}(x_1,\dots,x_N;t) 
\end{gather*}
and is subject to
\begin{gather}
 \mathbf{U}(x_1,\dots,x_{i-1},x_i,x_i,x_{i+2},\dots, x_N;t) = \mathbf{I}_2^{ \otimes (i-1)} \otimes \mathbf{B} \otimes \mathbf{I}_2^{ \otimes (N-i-1)}\nonumber\\
\qquad{} \times \mathbf{U}(x_1,\dots,x_{i-1},x_i,x_i+1,x_{i+2},\dots, x_N;t) \qquad \text{for all $i=1,\dots, N-1$}.\label{1038-31}
\end{gather}
 Then, it is possible to show that for each $(x_1,\dots,x_N) \in \mathbb{W}^N$, $\mathbf{U}(x_1,\dots,x_N;t)$ satisf\/ies the dif\/ferential equation for $\mathbf{P}_Y(x_1,\dots,x_N;t)$. For the initial positions of particles $Y = (y_1,\dots,y_N)$ $\in \mathbb{W}^N$, we put the Bethe Ansatz solution
\begin{gather}
\mathbf{U}(x_1,\dots,x_N;t) = \sum_{\sigma\in {S}_N}\mathbf{A}_{\sigma}\prod_{i=1}^N\big(\xi_{\sigma(i)}^{x_i-y_{\sigma(i)}-1}e^{\varepsilon(\xi_i) t}\big), \label{1032-31}
\end{gather}
where
\begin{gather*}
\varepsilon(\xi_i) = \frac{1}{\xi_i}-1
\end{gather*}
and
 $\xi_i$, $i=1,\dots, N$, are complex numbers with $0<|\xi_i|<1$, and $\mathbf{A}_{\sigma}$'s are $2^N \times 2^N$ matrices of complex numbers.

 \subsubsection[Matrices $\mathbf{A}_{\sigma}$]{Matrices $\boldsymbol{\mathbf{A}_{\sigma}}$} Let $T_i$, $i=1,\dots, N-1$, be simple transpositions in~$S_N$. It is well known that $T_1,\dots, T_{N-1}$ generate $S_N$ and satisfy the following relations \cite[Chapter~4]{Kassel-Turaev-2008}: for all $i,j=1,\dots, N-1$,
\begin{alignat}{3}
& T_iT_j = T_jT_i \qquad && \textrm{if} \quad |i - j| \geq 2,& \nonumber\\
&T_iT_jT_i = T_j T_iT_j \qquad && \textrm{if} \quad |i-j| = 1,& \nonumber\\
& T_i^2 = 1.&&&\label{1141-pm-523}
\end{alignat}
Substituting (\ref{1032-31}) into (\ref{1038-31}), we see that (\ref{1032-31}) satisf\/ies (\ref{1038-31}) provided that
\begin{gather}
\mathbf{A}_{T_i\sigma} =-\big(\mathbf{I}_{2^N} - \mathbf{I}_2^{ \otimes (i-1)} \otimes \mathbf{B} \otimes \mathbf{I}_2^{ \otimes (N-i-1)} \xi_{\beta}\big)^{-1} \big(\mathbf{I}_{2^N} - \mathbf{I}_2^{ \otimes (i-1)} \otimes \mathbf{B} \otimes \mathbf{I}_2^{ \otimes (N-i-1)} \xi_{\alpha}\big) \mathbf{A}_{\sigma}\nonumber \\
\hphantom{\mathbf{A}_{T_i\sigma}}{} = \big(\mathbf{I}_{2^{(i-1)}} \otimes \mathbf{S}_{\beta\alpha} \otimes \mathbf{I}_{2^{(N-i-1)}}\big)\mathbf{A}_{\sigma},\label{234-pm-523}
\end{gather}
where $\alpha =\sigma(i)$ and $\beta = \sigma(i+1)$ for all $\sigma$ and all $i=1,\dots, N-1$. Since $T_1,\dots, T_{N-1}$ generate~$S_N$, for each given $\sigma\in S_N$ there exists a f\/inite sequence $(a_i)_{i=1}^n$ of integers $1,\dots, N-1$ such that
\begin{gather}
\sigma = T_{a_n}\cdots T_{a_1}. \label{1140-pm-523}
\end{gather}
(Here, the representation (\ref{1140-pm-523}) is not unique because of (\ref{1141-pm-523}).) Suppose that $T_{a_i}$ in~(\ref{1140-pm-523}) interchanges~$\alpha$ and $\beta$, that is,
\begin{gather*}
T_{a_i}(\cdots \alpha\beta\cdots) = (\cdots \beta\alpha \cdots).
\end{gather*}
For each $T_{a_i}$ we def\/ine a $2^N \times 2^N$ matrix denoted by $\mathbf{T}_{a_i} =\mathbf{T}_{a_i}(\alpha,\beta) $ by the tensor product of $N-2$ identity matrices and $\mathbf{S}_{\beta\alpha}$ as follows.
\begin{Definition}\label{622-am}
Let $\mathbf{S}_{\beta\alpha}$ be given by (\ref{546pm511}). We def\/ine
\begin{gather*}
\mathbf{T}_{l}=\mathbf{T}_{l}({\alpha},{\beta}):= \mathbf{I}_2^{\otimes(l-1)} \otimes \mathbf{S}_{\beta\alpha} \otimes \mathbf{I}_2^{\otimes(N-l-1)} 
\end{gather*}
for $l=1,\dots,N-1$.
\end{Definition}

The matrices $\mathbf{T}_1,\dots, \mathbf{T}_{N-1}$ satisfy the matrix version of the relations (\ref{1141-pm-523}),
\begin{alignat}{3}
& \textrm{(i)} \quad && \mathbf{T}_i(\alpha,\beta)\mathbf{T}_j(\gamma,\delta) = \mathbf{T}_j(\gamma,\delta)\mathbf{T}_i(\alpha,\beta) \qquad \textrm{if} \quad |i - j| \geq 2,& \label{122-am-526}\\
& \textrm{(ii)} \quad &&\mathbf{T}_i(\beta,\gamma)\mathbf{T}_j(\alpha,\gamma)\mathbf{T}_i(\alpha,\beta) = \mathbf{T}_j(\alpha,\beta) \mathbf{T}_i(\alpha,\gamma)\mathbf{T}_j(\beta,\gamma) \qquad \textrm{if} \quad |i-j| = 1,& \label{123-am-526}\\
& \textrm{(iii)}\quad && \mathbf{T}_i(\beta,\alpha)\mathbf{T}_i(\alpha,\beta) = \mathbf{I}_{2^N}.&\label{124-am-526}
\end{alignat}
Here, (ii) and (iii) generalize (\ref{134-am-529}) and (\ref{135-am-529}), respectively.
\begin{Definition}
For given $\sigma=T_{a_n}\cdots T_{a_1}\in S_N$ we def\/ine
\begin{gather}
\mathbf{A}_{\sigma}:= \mathbf{T}_{a_n}\cdots \mathbf{T}_{a_1}. \label{148am511}
\end{gather}
Here, $\sigma$ may have a dif\/ferent representation $\sigma =T_{b_m}\cdots T_{b_1}$ but both $\mathbf{T}_{a_n}\cdots \mathbf{T}_{a_1}$ and $\mathbf{T}_{b_m}\cdots \mathbf{T}_{b_1}$ def\/ine the same matrix because of (\ref{122-am-526}), (\ref{123-am-526}) and~(\ref{124-am-526}). Thus, (\ref{148am511}) is well-def\/ined.
\end{Definition}

\begin{Lemma}If $\mathbf{A}_{\sigma}$ is given by \eqref{148am511}, then
\eqref{234-pm-523} holds. Hence \eqref{1032-31} satisfies \eqref{1038-31} for all $(x_1,\dots,x_N) \in \mathbb{Z}^N$.
\end{Lemma}
\begin{proof}
Let $\sigma = T_{a_n}\cdots T_{a_1}$ and $\mathbf{A}_{\sigma} = \mathbf{T}_{a_n}\cdots \mathbf{T}_{a_1}$. Suppose that $\sigma(i)= \alpha$ and $\sigma(i+1)=\beta$. Since $T_i\sigma = T_iT_{a_n}\cdots T_{a_1}$ and
\begin{gather*}
\mathbf{T}_{i}= \mathbf{I}_{2^{(i-1)}} \otimes \mathbf{S}_{\beta\alpha} \otimes \mathbf{I}_{2^{(N-i-1)}},
\end{gather*}
we immediately obtain
\begin{gather*}
\mathbf{A}_{T_i\sigma} = \mathbf{T}_{i}\mathbf{T}_{a_n}\cdots \mathbf{T}_{a_1}= \mathbf{T}_{i}\mathbf{A}_{\sigma} =\big(\mathbf{I}_{2^{(i-1)}} \otimes \mathbf{S}_{\beta\alpha} \otimes \mathbf{I}_{2^{(N-i-1)}}\big)\mathbf{A}_{\sigma}.\tag*{\qed}
\end{gather*}\renewcommand{\qed}{}
\end{proof}

Finally, if we integrate (\ref{1032-31}) over circles centered at the origin with radii less than 1, we obtain
\begin{gather}\label{1223-am-519}
\mathbf{P}_Y(X;t) = \dashint_C\cdots \dashint_C\sum_{\sigma\in {S}_N}\mathbf{A}_{\sigma}\prod_{i=1}^N\big(\xi_{\sigma(i)}^{x_i-y_{\sigma(i)}-1}e^{\varepsilon(\xi_i) t}\big) {\rm d}\xi_1\cdots {\rm d}\xi_N.
\end{gather}
The proof of
\begin{gather*}
\mathbf{P}_Y(X;0) = \begin{cases} \mathbf{I}_{2^N}& \textrm{if}~X=Y,\\
\mathbf{0}& \textrm{if}~X\neq Y
\end{cases}
\end{gather*}
is given in \cite{Tracy-Widom-2013}.

\section{Proofs of the main results}\label{section3}
\subsection{Proof of Theorem \ref{142-am-518}}

To prove Theorem \ref{142-am-518}, we need to sum $P_{(Y,21\dots 1)}(X,21\dots 1;t)$ over all allowed conf\/igurations, that is, we need to compute
\begin{gather}
\mathbb{P}_{(Y,21\dots 1)}(E_t) = \sum_{x=x_1<\cdots<x_N} P_{(Y,21\dots 1)}(X,21\dots 1;t).\label{1219-am-520}
\end{gather}
Putting $x_1=x,x_2=x+i_1,\dots,x_N = x+i_1+\cdots +i_{N-1}$, (\ref{1219-am-520}) becomes
\begin{gather*}
\sum_{i_1,\dots,i_{N-1}=1}^{\infty} \dashint_C\cdots \dashint_C\sum_{\sigma\in S_N}\big[\mathbf{A}_{\sigma}\big]\big(\xi_{\sigma(2)}\cdots\xi_{\sigma(N)}\big)^{i_1}\big(\xi_{\sigma(3)}\cdots\xi_{\sigma(N)}\big)^{i_2}\cdots \big(\xi_{\sigma(N)}\big)^{i_N} \nonumber \\
\qquad\quad{} \times \prod_{i=1}^N\big(\xi_{i}^{x-y_{i}-1}e^{\varepsilon(\xi_i) t}\big) {\rm d}\xi_1\cdots {\rm d}\xi_N \\
\qquad{} = \dashint_C\cdots \dashint_C\sum_{\sigma\in S_N}\big[\mathbf{A}_{\sigma}\big]\frac{\xi_{\sigma(2)}\xi_{\sigma(3)}^{2}\cdots \xi_{\sigma_{N}}^{N-1}}{(1-\xi_{\sigma(2)}\cdots\xi_{\sigma(N)})(1-\xi_{\sigma(3)}\cdots\xi_{\sigma(N)})\cdots (1-\xi_{\sigma(N)})}\\
\qquad\quad{} \times \prod_{i=1}^N\big(\xi_{i}^{x-y_{i}-1}e^{\varepsilon(\xi_i) t}\big) {\rm d}\xi_1\cdots {\rm d}\xi_N,
\end{gather*}
where $[\mathbf{A}_{\sigma}]$ is the $(2^{N-1}+1,2^{N-1}+1)^\textrm{th}$ element of $\mathbf{A}_{\sigma}$. (We write $[\mathbf{A}]_{i,j}$ for the $(i,j)^{\textrm{th}}$ element of matrix $\textbf{A}$ but we will simply write $[\mathbf{A}]$ for $i,j=2^{N-1}+1$.) In order to compute
\begin{gather}\label{1221-am-520}
 \sum_{\sigma\in S_N}[\mathbf{A}_{\sigma}]\frac{\xi_{\sigma(2)}\xi_{\sigma(3)}^{2}\cdots \xi_{\sigma_{N}}^{N-1}}{(1-\xi_{\sigma(2)}\cdots\xi_{\sigma(N)})(1-\xi_{\sigma(3)}\cdots\xi_{\sigma(N)})\cdots (1-\xi_{\sigma(N)})},
\end{gather}
f\/irst, we will f\/ind an expression for $[\mathbf{A}_{\sigma}]$.
\begin{Lemma}\label{420pm512}
Let $\mathbf{T}_{l}=\mathbf{T}_{l}(\alpha,\beta)$ be a $2^N \times 2^N$ matrix in Definition~{\rm \ref{622-am}} and let $\mathbf{A}_{\sigma}$ be given by~\eqref{148am511}.
\begin{itemize}\itemsep=0pt
\item [$(a)$] If $l\neq 1$, then $[\mathbf{T}_{l}(\alpha,\beta)] =S_{\beta\alpha},$ and if $l=1$, then $[\mathbf{T}_{l}(\alpha,\beta)] =-1$.
\item [$(b)$] Let $(a_1,\dots,a_{2^N})$ be the $(2^{N-1}+1)^{\textrm{th}}$ row vector of $\mathbf{T}_{l}$. Then, $a_k=0$ for all $k \neq 2^{N-1}+1$.
\item [$(c)$] $\mathbf{T}_{l}$ is an upper-triangular matrix.
\item [$(d)$] The diagonal terms of $\mathbf{A}_{\sigma}$ are given by
\begin{gather*}
[\mathbf{A}_{\sigma}]_{l,l} = \begin{cases}
[\mathbf{T}_{a_n}]_{l,l}\cdots[\mathbf{T}_{a_2}]_{l,l} [\mathbf{T}_{a_1}]_{l,l} & \textrm{if}~\sigma \neq 1, \\
1 & \textrm{if}~\sigma = 1.
\end{cases}
\end{gather*}
\end{itemize}
\end{Lemma}

\begin{proof}
All these properties are easily verified by observing the forms of $\mathbf{T}_{l}$.
\end{proof}
\begin{Remark}
Unlike Lemma \ref{420pm512}(d), to the best of the author's knowledge, it is nontrivial to f\/ind the non-diagonal terms of $\mathbf{A}_{\sigma}$ without directly performing matrices multiplication $\mathbf{T}_{a_n}\cdots\mathbf{T}_{a_1}$. In general, this makes it inconvenient to f\/ind an explicit form of $P_{(Y,\nu)}(X,\pi;t)$ for $\pi\neq \nu$.
\end{Remark}

Now, we give an expression for $[\mathbf{A}_{\sigma}]$.
\begin{Lemma}\label{146-am-520} If $N \geq 2$,
\begin{gather}\label{336-pm-524}
[\mathbf{A}_{\sigma}] =\operatorname{sgn}(\sigma)\prod_{i=0}^{N-2}\left(\frac{1-\xi_{2+i}}{1-\xi_{\sigma(2+i)}} \right)^i.
\end{gather}
\end{Lemma}
\begin{proof}
If $\sigma = 1$, the statement is trivial. Suppose that $\sigma \neq 1$. We prove the statement by induction on $N$. When $N=2$, it is easy to verify the statement if we observe (\ref{102224}). Let $\sigma' \in S_N$ and suppose that (\ref{336-pm-524}) holds for $N$, that is,
\begin{gather*}
[\mathbf{A}_{\sigma'}] =
\operatorname{sgn}(\sigma')\left(\frac{1-\xi_{3}}{1-\xi_{\sigma'(3)}} \right)\cdots\left(\frac{1-\xi_{N-1}}{1-\xi_{\sigma'(N-1)}} \right)^{N-3}\left(\frac{1-\xi_{N}}{1-\xi_{\sigma'(N)}} \right)^{N-2}
\end{gather*}
holds. For any $\sigma \in S_{N+1}$, there are $\sigma' \in S_{N}$ and an integer $K= \{1,\dots, N+1\}$ such that $\sigma(i) = \sigma'(i)$ for $1\leq i\leq K-1$, $\sigma(K) = N+1$ and $\sigma(i) = \sigma'(i-1)$ for $K+1 \leq i \leq N+1$. Suppose that $\sigma' = T_{a_n}\cdots T_{a_1}$ for some f\/inite sequence $a_1,\dots,a_n$ taking values in $\{1,\dots, N-1\}$. If we view $T_{a_n},\dots, T_{a_1}$ as simple transpositions in $S_{N+1}$, then
\begin{gather*}
\sigma = T_KT_{K+1}\cdots T_NT_{a_n}\cdots T_{a_1}\big(1~2\cdots N~(N+1)\big)
\end{gather*}
and
\begin{gather*}
[\mathbf{A}_{\sigma}]_{2^N+1,2^N+1}= [\mathbf{T}_K]_{2^N+1,2^N+1} \cdots[\mathbf{T}_N]_{2^N+1,2^N+1} [\mathbf{T}_{a_n}\cdots\mathbf{T}_{a_1}]_{2^N+1,2^N+1} \\
\hphantom{[\mathbf{A}_{\sigma}]_{2^N+1,2^N+1}}{} = [\mathbf{T}_K]_{2^N+1,2^N+1} \cdots[\mathbf{T}_N]_{2^N+1,2^N+1} [\mathbf{A}_{\sigma'}]
\end{gather*}
by Lemma \ref{420pm512}(d). Using the induction hypothesis and Lemma \ref{420pm512}(a), we obtain
\begin{gather}
[\mathbf{A}_{\sigma}]_{2^N+1,2^N+1}= \left(-\frac{1-\xi_{N+1}}{1-\xi_{\sigma'(K)}}\right)\left(-\frac{1-\xi_{N+1}}{1-\xi_{\sigma'(K+1)}}\right)\cdots \left(-\frac{1-\xi_{N+1}}{1-\xi_{\sigma'(N)}}\right) \nonumber\\
\hphantom{[\mathbf{A}_{\sigma}]_{2^N+1,2^N+1}=}{} \times\operatorname{sgn}(\sigma')\left(\frac{1-\xi_{3}}{1-\xi_{\sigma'(3)}} \right)\cdots\left(\frac{1-\xi_{N-1}}{1-\xi_{\sigma'(N-1)}} \right)^{N-3}\left(\frac{1-\xi_{N}}{1-\xi_{\sigma'(N)}} \right)^{N-2}.\label{336-am-526}
\end{gather}
If we recall that $\sigma(i) = \sigma'(i)$ for $1\leq i\leq K-1$, $\sigma(K) = N+1$ and $\sigma(i) = \sigma'(i-1)$ for $K+1 \leq i \leq N+1$ and note that $(-1)^{N-K+1}\operatorname{sgn}(\sigma')=\operatorname{sgn}(\sigma)$, then we see that~(\ref{336-am-526}) is equal to
\begin{gather*}
\operatorname{sgn}(\sigma)\left(\frac{1-\xi_{3}}{1-\xi_{\sigma(3)}}\right)\left(\frac{1-\xi_{4}}{1-\xi_{\sigma(4)}}\right)^2\cdots
\left(\frac{1-\xi_{N+1}}{1-\xi_{\sigma(N+1)}}\right)^{N-1}.\tag*{\qed}
\end{gather*}\renewcommand{\qed}{}
\end{proof}

Now, let us come back to the expression (\ref{1221-am-520}). The sum in (\ref{1221-am-520}) with $[\mathbf{A}_{\sigma}]$ in (\ref{336-pm-524}) will be simplif\/ied by the algebraic identity
\begin{gather}
 \sum_{\sigma\in S_N}[\mathbf{A}_{\sigma}]\frac{\xi_{\sigma(2)}\xi_{\sigma(3)}^{2}\cdots \xi_{\sigma_{N}}^{N-1}}{(1-\xi_{\sigma(2)}\cdots\xi_{\sigma(N)})(1-\xi_{\sigma(3)}\cdots\xi_{\sigma(N)})\cdots (1-\xi_{\sigma(N-1)}\xi_{\sigma(N)})(1-\xi_{\sigma(N)})}\nonumber\\
\qquad = (1-\xi_1)\prod_{1\leq i<j\leq N}\frac{\xi_j-\xi_i}{1-\xi_i}\prod_{i=1}^N\frac{1}{1-\xi_i},\qquad N\geq 2.\label{1221-am-525}
\end{gather}
Once this identity is proved, the proof of Theorem \ref{142-am-518} is completed.
\begin{Remark}
Although this paper does not deal with the ASEP with second class particles, it is believed that the generalization of the identity (\ref{1221-am-525}) to the ASEP with second class particles is possible. This generalization was conf\/irmed for small systems by the author. Since $P_{(Y,\nu)}(X,\nu;t)$ in the ASEP with second class particles is not determinantal, it seems that the determinantal structure of $P_{(Y,\nu)}(X,\nu;t)$ in the TASEP with second class particles is not essential in deriving~(\ref{1221-am-525}). But, we do not have a result corresponding to Lemma~\ref{146-am-520} for the ASEP with second class particles.
\end{Remark}

If we observe that
\begin{gather*}
\frac{1-\xi_1}{\prod\limits_{i=0}^{N-2}(1-\xi_{2+i})^{i}}\frac{1}{\prod\limits_{i<j}(1-\xi_i)} = \prod_{i=1}^N\frac{1}{(1-\xi_i)^{N-2}},
\end{gather*}
then we see that (\ref{1221-am-525}) is equivalent to
\begin{gather}
\sum_{\sigma\in S_N}\operatorname{sgn}(\sigma)\frac{1}{(1-\xi_{\sigma(3)})(1-\xi_{\sigma(4)})^2\cdots(1-\xi_{\sigma(N)})^{N-2}}\nonumber\\
\qquad\quad{} \times \frac{\xi_{\sigma(2)}\xi_{\sigma(3)}^2\xi_{\sigma(4)}^3\cdots \xi_{\sigma(N)}^{N-1}}{(1-\xi_{\sigma(2)}\cdots\xi_{\sigma(N)})(1-\xi_{\sigma(3)}\cdots\xi_{\sigma(N)})\cdots (1-\xi_{\sigma(N-1)}\xi_{\sigma(N)})(1-\xi_{\sigma(N)})} \nonumber\\
 \qquad {} =\prod_{i=1}^N\frac{1}{(1-\xi_i)^{N-1}}\prod_{1\leq i<j\leq N}(\xi_j-\xi_i), \qquad N \geq 2.\label{301-pm-520}
\end{gather}
Also, if we substitute $1/\xi_{N-i+1}$ for $\xi_i $ in (\ref{301-pm-520}), we obtain another equivalent identity
\begin{gather}
\sum_{\sigma\in S_N}\operatorname{sgn}(\sigma)\frac{1}{(\xi_{\sigma(1)}-1)^{N-2}(\xi_{\sigma(2)}-1)^{N-3}\cdots(\xi_{\sigma(N-2)}-1)}\nonumber\\
\qquad\quad{} \times \frac{\xi_{\sigma(N-2)}\xi_{\sigma(N-3)}^2\cdots \xi_{\sigma(1)}^{N-2}}{(\xi_{\sigma(1)}\cdots\xi_{\sigma(N-1)}-1)(\xi_{\sigma(1)}\cdots\xi_{\sigma(N-2)}-1)\cdots (\xi_{\sigma(1)}\xi_{\sigma(2)}-1)(\xi_{\sigma(1)}-1)}\nonumber \\
\qquad{} =\prod_{i=1}^N\frac{1}{(\xi_i-1)^{N-1}}\prod_{1\leq i<j\leq N}(\xi_j-\xi_i).\label{530-pm-520}
\end{gather}
We will prove the identity (\ref{530-pm-520}). In order to prove (\ref{530-pm-520}), we will use an identity for the TASEP. In (3.2) with $p=1$ in \cite{Tracy-Widom-2008}, $A_{\sigma}$ can be written as
\begin{gather*}
A_{\sigma}=\operatorname{sgn}(\sigma)\left(\frac{1-\xi_2}{1-\xi_{\sigma(2)}} \right)\left(\frac{1-\xi_3}{1-\xi_{\sigma(3)}} \right)^2\cdots\left(\frac{1-\xi_N}{1-\xi_{\sigma(N)}} \right)^{N-1}.
\end{gather*}
Using (1.6) with $p=1$ in \cite{Tracy-Widom-2008}, we have
\begin{gather*}
 \sum_{\sigma\in S_N}A_{\sigma}\frac{\xi_{\sigma(2)}\xi_{\sigma(3)}^2\xi_{\sigma(4)}^3\cdots \xi_{\sigma(N)}^{N-1}}{(1-\xi_{\sigma(2)}\cdots\xi_{\sigma(N)})(1-\xi_{\sigma(3)}\cdots\xi_{\sigma(N)})\cdots (1-\xi_{\sigma(N-1)}\xi_{\sigma(N)})(1-\xi_{\sigma(N)})} \nonumber\\
 \qquad{} =(1-\xi_1\cdots\xi_N)\prod_{i=1}^N\frac{1}{1-\xi_i}\prod_{1\leq i<j\leq N}\frac{\xi_j-\xi_i}{1-\xi_i},\qquad N\geq 2 ,
\end{gather*}
equivalently,
\begin{gather}
\sum_{\sigma\in S_N}\operatorname{sgn}(\sigma)\frac{1}{(1-\xi_{\sigma(2)})(1-\xi_{\sigma(3)})^2\cdots(1-\xi_{\sigma(N)})^{N-1}}\nonumber\\
\qquad\quad{} \times \frac{\xi_{\sigma(2)}\xi_{\sigma(3)}^2\xi_{\sigma(4)}^3\cdots \xi_{\sigma(N)}^{N-1}}{(1-\xi_{\sigma(2)}\cdots\xi_{\sigma(N)})(1-\xi_{\sigma(3)}\cdots\xi_{\sigma(N)})\cdots (1-\xi_{\sigma(N-1)}\xi_{\sigma(N)})(1-\xi_{\sigma(N)})}\nonumber \\
 \qquad{} = (1-\xi_1\cdots\xi_N)\prod_{i=1}^N\frac{1}{(1-\xi_i)^{N}}\prod_{1\leq i<j\leq N}(\xi_j-\xi_i),\qquad N\geq 2 .\label{308-am-521}
\end{gather}
If we substitute $1/\xi_{N-i+1}$ for $\xi_i $ in (\ref{308-am-521}), we obtain an equivalent version of (\ref{308-am-521}),
\begin{gather}
\sum_{\sigma\in S_N}\operatorname{sgn}(\sigma)\frac{1}{(\xi_{\sigma(1)}-1)^{N-1}(\xi_{\sigma(2)}-1)^{N-2}\cdots(\xi_{\sigma(N-1)}-1)}\nonumber\\
\qquad\quad{} \times \frac{\xi_{\sigma(N-1)}\xi_{\sigma(N-2)}^2\cdots \xi_{\sigma(1)}^{N-1}}{(\xi_{\sigma(1)}\cdots\xi_{\sigma(N-1)}-1)(\xi_{\sigma(1)}\cdots\xi_{\sigma(N-2)}-1)\cdots (\xi_{\sigma(1)}\xi_{\sigma(2)}-1)(\xi_{\sigma(1)}-1)} \nonumber\\
 \qquad{} = (\xi_1\cdots \xi_N-1)\prod_{i=1}^N\frac{1}{(\xi_i-1)^{N}}\prod_{1\leq i<j\leq N}(\xi_j-\xi_i).\label{631-pm-520}
\end{gather}
This is the identity we will use in order to prove (\ref{530-pm-520}). Also, we need the following result to prove (\ref{530-pm-520}).
\begin{Lemma}[Vandermonde determinants]\label{827-am-521}
\begin{gather}
\sum_{\alpha=1}^N(-1)^{N+\alpha}(\xi_{\alpha}-1)^{N-1}\prod_{\substack{1\leq i < j \leq N\\ i,j \neq \alpha}}(\xi_j - \xi_i) = \prod_{1\leq i<j\leq N}(\xi_j - \xi_i). \label{855-am-521}
\end{gather}
\end{Lemma}
\begin{proof}
The left hand side of (\ref{855-am-521}) is a cofactor expansion of
\begin{gather*}
\det\left[
 \begin{matrix}
 1 &1 & \cdots & 1 \\
 \xi_1 & \xi_2 & \cdots & \xi_N \\
 \vdots & \vdots & & \vdots\\
 \xi_1^{N-2} & \xi_2^{N-2} & \cdots & \xi_N^{N-2} \\
 (\xi_1-1)^{N-1} & (\xi_2-1)^{N-1} & \cdots& (\xi_N-1)^{N-1}\\
 \end{matrix}
\right].
\end{gather*}
Note that
\begin{gather*}
(\xi_{\alpha}-1)^{N-1}= \sum_{k=1}^Na_{N-k}\xi_{\alpha}^{N-k}
\end{gather*}
for some constants $a_{N-1},\dots, a_{0}$. (Here, $a_{N-1}=1$.) If we add a multiple of the f\/irst row by the constant $-a_0$ to the $N$th row, the determinant does not change and is equal to
\begin{gather*}
\det\left[
 \begin{matrix}
 1 &1 & \cdots & 1 \\
 \xi_1 & \xi_2 & \cdots & \xi_N \\
 \vdots & \vdots & & \vdots\\
 \xi_1^{N-2} & \xi_2^{N-2} & \cdots & \xi_N^{N-2} \\
 \sum\limits_{k=1}^{N-1}a_{N-k}\xi_{1}^{N-k} & \sum\limits_{k=1}^{N-1}a_{N-k}\xi_{2}^{N-k} & \cdots& \sum\limits_{k=1}^{N-1}a_{N-k}\xi_{N}^{N-k}\\
 \end{matrix}
\right].
\end{gather*}
We successively perform these row operations to obtain
\begin{gather*}
\det\left[
 \begin{matrix}
 1 &1 & \cdots & 1 \\
 \xi_1 & \xi_2 & \cdots & \xi_N \\
 \vdots & \vdots & & \vdots\\
 \xi_1^{N-2} & \xi_2^{N-2} & \cdots & \xi_N^{N-2} \\
 \xi_{1}^{N-1} & \xi_{2}^{N-2} & \cdots& \xi_{N}^{N-1}\\
 \end{matrix}
\right],
\end{gather*}
which is the Vadermonde determinant.
\end{proof}

Now, we prove (\ref{530-pm-520}).
\begin{proof}[Proof of identity (\ref{530-pm-520})]
 When $N=2$, it is easy to show (\ref{530-pm-520}). We will show (\ref{530-pm-520}) for $N\geq 3$. The left hand side of (\ref{530-pm-520}) is an antisymmetric function of $\xi_1,\dots, \xi_{N}$ because the sum is an anti-symmetrized sum, so it is divisible by the Vandermonde determinant. Hence, the left hand side of (\ref{530-pm-520}) can be written as
\begin{gather*}
G(\xi_1,\dots, \xi_{N})\times \prod_{1\leq i<j\leq N}(\xi_j-\xi_i)
\end{gather*}
where $G(\xi_1,\dots, \xi_{N})$ is a symmetric function of $\xi_1,\dots, \xi_{N}$. So, we want to show that
\begin{gather*}
G(\xi_1,\dots, \xi_{N})=\prod_{i=1}^N\frac{1}{(\xi_i-1)^{N-1}}.
\end{gather*}
 Fix $\alpha \in \{1,\dots, N\}$. Let $\sigma'$ be a bijective mapping from $\{1,\dots, N-1\}$ onto $\{1,\dots, N\}\setminus \{\alpha\}$ and let $S'_{N-1}$ be the set of $\sigma'$. Let $\sigma_{\alpha}\in S_{N}$ be a permutation such that $\sigma_{\alpha}(N)=\alpha$ and $\sigma_{\alpha}(i) = \sigma'(i)$ for $i=1,\dots,N-1$. Then, the left hand side of (\ref{530-pm-520}) is equal to
\begin{gather}
 \sum_{\alpha=1}^N\sum_{\sigma_{\alpha}\in S_N}\operatorname{sgn}(\sigma_{\alpha})\frac{1}{(\xi_{\sigma_{\alpha}(1)}-1)^{N-2}(\xi_{\sigma_{\alpha}(2)}-1)^{N-3} \cdots(\xi_{\sigma_{\alpha}(N-2)}-1)}\nonumber\\
\qquad{} \times \frac{\xi_{\sigma_{\alpha}(N-2)}\xi_{\sigma_{\alpha}(N-3)}^2\cdots \xi_{\sigma_{\alpha}(1)}^{N-2}}{(\xi_{\sigma_{\alpha}(1)}\cdots\xi_{\sigma_{\alpha}(N-1)}-1)(\xi_{\sigma_{\alpha}(1)}\cdots\xi_{\sigma_{\alpha}(N-2)}-1)\cdots (\xi_{\sigma_{\alpha}(1)}-1)}.\label{745-am-521}
\end{gather}
The sum over $\sigma_{\alpha} \in S_{N}$ for a f\/ixed $\alpha$ in (\ref{745-am-521}) is equal to
\begin{gather}
\sum_{\sigma' \in S'_{N-1}} (-1)^{N-\alpha}\operatorname{sgn}(\sigma')\frac{1}{(\xi_{\sigma'(1)}-1)^{N-2}(\xi_{\sigma'(2)}-1)^{N-3}\cdots(\xi_{\sigma'(N-2)}-1)}\nonumber\\
\qquad{} \times \frac{\xi_{\sigma'(N-2)}\xi_{\sigma'(N-3)}^2\cdots \xi_{\sigma'(1)}^{N-2}}{(\xi_{\sigma'(1)}\cdots\xi_{\sigma'(N-1)}-1)(\xi_{\sigma'(1)}\cdots\xi_{\sigma'(N-2)}-1)\cdots (\xi_{\sigma'(1)}-1)},\label{745-am-525}
\end{gather}
because $\sigma_{\alpha}(i) = \sigma'(i)$ for $i=1,\dots,N-1$ and $\operatorname{sgn}(\sigma_{\alpha}) = (-1)^{N-\alpha}\operatorname{sgn}(\sigma')$. Therefore,
if we apply (\ref{631-pm-520}) for $N-1$ to (\ref{745-am-525}) and sum over $\alpha$, we obtain
\begin{gather*}
\sum_{\alpha=1}^N(-1)^{N-\alpha} \prod_{\substack{i=1\\i\neq \alpha }}^{N-1}\frac{1}{(\xi_i-1)^{N-1}}\prod_{\substack{1\leq i<j\leq {N-1}\\ i,j\neq \alpha}}(\xi_j-\xi_i)= \prod_{i=1}^{N}\frac{1}{(\xi_i-1)^{N-1}}\prod_{1\leq i<j\leq {N}}(\xi_j-\xi_i)
\end{gather*}
by using Lemma \ref{827-am-521}.
\end{proof}

\subsection{Proof of Theorem \ref{144-am-518} and Corollary \ref{143-am-519}}
\begin{Lemma}\label{1123-pm-529}
Let $k_i$ be any integer such that $0 \leq k_i \leq i-2$ for $i=2,\dots,N$, and $l$ be any nonnegative integer. Let us consider
\begin{gather}\label{102-am-530}
 \det\left[
 \begin{matrix}
 \xi_1^{N-1+l} & \xi_1^{N-2 +k_2} & \xi_1^{N-3+k_3} & \xi_1^{N-4+k_4} & \cdots & \xi_1^{1+k_{N-1}} & \xi_1^{k_N} \\
 \xi_2^{N-1+l} & \xi_2^{N-2+k_2} & \xi_2^{N-3+k_3} & \xi_2^{N-4+k_4}& \cdots & \xi_2^{1+k_{N-1}} & \xi_2^{k_N} \\
 \vdots & \vdots & \vdots & \vdots & \vdots& \vdots & \vdots\\
 \xi_N^{N-1+l} & \xi_N^{N-2+k_2} & \xi_N^{N-3+k_3}& \xi_N^{N-4+k_4}& \cdots & \xi_N^{1+k_{N-1}}& \xi_N^{k_N} \\
 \end{matrix}
\right].
\end{gather}
\begin{itemize}\itemsep=0pt
\item [$(a)$] If there is a nonzero $k_i$ for some $i=3,\dots,N$, the determinant is zero.
\item [$(b)$] If $k_i = 0$ for all $i$, the determinant is
\begin{gather*}
h_l(\xi_1,\dots,\xi_N)\prod_{1\leq i<j\leq N}(\xi_i-\xi_j).
\end{gather*}
\end{itemize}
\end{Lemma}
\begin{proof}
(a) Suppose that $k_{\alpha} \neq 0$ for some $\alpha \geq 3$ but the determinant is nonzero. Since the power of the terms in the second column is $(N-2)$ and the determinant should be nonzero, the power of the terms in the third column must be $(N-3).$ If we repeat this process, we see that the power of the terms in the $i^{\textrm{th}}$ column should be $(N-i)$, $i=2,\dots, \alpha-1$, for the determinant to be nonzero. If $k_{\alpha} \neq 0$, then the power of the terms in the ${\alpha}^{\textrm{th}}$ column is one of $N-\alpha+ k_{\alpha}$ where $k_{\alpha} = 1,2,\dots, \alpha -2$, and hence the determinant is zero, which is a contradiction.

(b) If $k_i = 0$ for all $i$, then we have
\begin{gather*}
 \det\left[
 \begin{matrix}
 \xi_1^{N-1+l} & \xi_1^{N-2} & \xi_1^{N-3} & \cdots & 1 \\
 \xi_2^{N-1+l} & \xi_2^{N-2} & \xi_2^{N-3} &\cdots & 1 \\
 \vdots & \vdots & \vdots & \vdots & \vdots\\
 \xi_N^{N-1+l} & \xi_N^{N-2} & \xi_N^{N-3}& \cdots & 1 \\
 \end{matrix}
\right] = \det\big[\xi_i^{\lambda_{j}+N-j}\big]_{i,j=1}^N = h_l\prod_{1\leq i<j\leq N}(\xi_i-\xi_j)
\end{gather*}
where $\lambda = (\lambda_1,\lambda_2,\dots,\lambda_N )= (l,0,\dots,0)$ is a partition and $h_l= h_l(\xi_1,\dots,\xi_N)$ is the complete symmetric polynomial of degree~$l$ \cite[Chapter I.3]{Macdonald}.
\end{proof}
Now, we apply the initial condition (\ref{107-am-518}) to (\ref{125-am-518}). Then, we have, after some manipulations,
\begin{gather*}\label{531-pm-521}
\mathbb{P}_{(Y,21\dots 1)}(E_t) = \dashint_C\cdots\dashint_C\prod_{1\leq i<j\leq N}(\xi_j-\xi_i)\prod_{i=1}^N\frac{1}{(1-\xi_i)^{N-1}}\prod_{i=1}^N \big(\xi_i^{x-N-1-l}e^{\varepsilon(\xi_i) t}\big)\\
\qquad\quad{} \times\left(\prod_{i=0}^{N-2}(1-\xi_{2+i})^i\right)\xi_{1}^{N+l-1}\xi_{2}^{N-2}\xi_{3}^{N-3}\cdots \xi_{N-1}{\rm d}\xi_1\cdots {\rm d}\xi_N \\
\qquad{} = \frac{1}{N!} \dashint_C\cdots\dashint_C\prod_{i<j}(\xi_j-\xi_i)\prod_{i=1}^N\frac{1}{(1-\xi_i)^{N-1}}\prod_{i=1}^N \big(\xi_i^{x-N-1-l}e^{\varepsilon(\xi_i) t}\big)\\
 \qquad\quad{} \times\left[\sum_{\sigma\in S_N}\operatorname{sgn}(\sigma)\left(\prod_{i=0}^{N-2}(1-\xi_{\sigma(2+i)})^i\right)\xi_{\sigma(1)}^{N+l-1}\xi_{\sigma(2)}^{N-2}\xi_{\sigma(3)}^{N-3}\cdots \xi_{\sigma(N-1)}\right]{\rm d}\xi_1\cdots {\rm d}\xi_N.
\end{gather*}
If we expand $\prod\limits_{i=0}^{N-2}(1-\xi_{\sigma(2+i)})^i$, then we see that each term is in the form of
\begin{gather*}
C\cdot\xi_{\sigma(3)}^{k_3}\xi_{\sigma(4)}^{k_4}\cdots\xi_{\sigma(N)}^{k_N}
\end{gather*}
where $k_i$ are some integers such that $0\leq k_i \leq i-2$ and $C$ is some constant, and
\begin{gather*}
\sum_{\sigma\in S_N}\operatorname{sgn}(\sigma)\xi_{\sigma(1)}^{N-1+l}\xi_{\sigma(2)}^{N-2}\xi_{\sigma(3)}^{N-3+k_3}\cdots \xi_{\sigma(N-1)}^{1+k_{N-1}}\xi_{\sigma(N)}^{k_N}
\end{gather*}
is the determinant (\ref{102-am-530}). Hence, by Lemma \ref{1123-pm-529}, we obtain
\begin{gather}
 \mathbb{P}_{(Y,21\dots 1)}(E_t) = \frac{(-1)^{N(N-1)/2}}{N!} \dashint_C\cdots\dashint_Ch_l(\xi_1,\dots,\xi_N)\prod_{1\leq i<j\leq N}(\xi_j-\xi_i)^2\prod_{i=1}^N\frac{1}{(\xi_i-1)^{N-1}}\nonumber\\
\hphantom{\mathbb{P}_{(Y,21\dots 1)}(E_t) =}{} \times\prod_{i=1}^N \big(\xi_i^{x-N-l-1}e^{\varepsilon(\xi_i) t}\big){\rm d}\xi_1\cdots {\rm d}\xi_N.\label{152-pm-84}
\end{gather}
Finally, we immediately obtain (\ref{530-pm-517}) if $l=0$ in (\ref{152-pm-84}).

\subsection*{Acknowledgements}
This work was supported by the social policy grant from Nazarbayev University. The author is grateful to the anonymous referees for valuables comments and suggestions.

\pdfbookmark[1]{References}{ref}
\LastPageEnding

\end{document}